\documentclass{article}
\usepackage[utf8]{inputenc}
\usepackage[T1]{fontenc}
\usepackage[oldstyle]{fbb}

\usepackage{amssymb}
\usepackage{stmaryrd}
\usepackage{amsmath,amsfonts}
\usepackage[tikz]{bclogo}
\usepackage[top=4.34cm]{geometry}

\usepackage{lmodern}
\usepackage{textcomp}
\usepackage{mathtools, bm}
\usepackage{amssymb, bm}
\usepackage[unq]{unique}
\usepackage{enumerate}
\usepackage{comment}
\usepackage[breaklinks]{hyperref}

\usepackage[numbers]{natbib}  

\usepackage[breaklinks]{hyperref}
\usepackage[numbers]{natbib}

\usepackage{amsthm}

\usepackage{cleveref}
\usepackage[utf8]{inputenc}
\usepackage[T1]{fontenc}
\usepackage{enumitem}
\usepackage{tipa}
\usepackage{float}
\usepackage{dirtytalk}
\usepackage{pbox}
\usepackage{xcolor}
\usepackage{setspace}
\usepackage[normalem]{ulem}
\usepackage{geometry}
\usepackage{amsmath}
\usepackage{amsthm}
\usepackage{tikz}
\usepackage{amssymb}
\usepackage{multicol}
\usepackage{microtype}
\usepackage{datetime}
\usepackage{endnotes}
\usepackage{babel}
\usepackage[mathscr]{euscript}
\usepackage{bigfoot} 
\usetikzlibrary {arrows.meta}
\newcommand{\personal}[1]{}
\newtheorem{theorem}{Theorem}[section]
\newtheorem*{theorem*}{Theorem}
\newtheorem{claim}[theorem]{Claim}

\newtheorem{lemma}[theorem]{Lemma}

\newtheorem{question}[theorem]{Question}

\DeclareMathOperator{\BT}{BT}

\title{Completely independent spanning trees in the hypercube}
\author{Benedict Randall Shaw}
\date{December 2024}

\begin{document}

\maketitle

\begin{abstract}
    We say two spanning trees of a graph are completely independent if their edge sets are disjoint, and for each pair of vertices, the paths between them in each spanning tree do not have any other vertex in common. Pai and Chang \cite{pai_constructing_2016} constructed two such spanning trees in the hypercube \(Q_n\) for sufficiently large \(n\), while Kandekar and Mane \cite{kandekar_constructing_2024} recently showed there are \(3\) pairwise completely independent spanning trees in hypercubes \(Q_n\) for sufficiently large \(n\). We prove that for each \(k\), there exist \(k\) completely independent spanning trees in \(Q_n\) for sufficiently large \(n\). In fact, we show that there are \((\frac1{12}+o(1))n\) spanning trees in \(Q_n\), each with diameter \((2+o(1))n\). As the minimal diameter of any spanning tree of \(Q_n\) is \(2n-1\), this diameter is asymptotically optimal. We prove a similar result for the powers \(H^n\) of any fixed graph \(H\).
\end{abstract}

\section{Introduction}

Let \(G\) be a graph. We say two spanning trees of \(G\) are \emph{edge-disjoint} if they have no edges in common. If two spanning trees are edge-disjoint, and additionally for any two vertices \(u,v\) of \(G\), the paths from \(u\) to \(v\) within each spanning tree do not meet at any other vertices, then we call the trees \emph{completely independent}.

The study of completely independent spanning trees has applications in fault-tolerant communication protocols and secure message distribution \cite{pai_three_2020}, because finding \(k\) such trees allows information to be sent reliably even when \(k-1\) nodes have failed. 

The \textit{hypercube} \(Q_n\) has vertices \(\{0,1\}^n\), and edges those pairs of vertices which differ in exactly one coordinate. Pai and Chang \cite{pai_constructing_2016} constructed two completely independent spanning trees in hypercubes of dimension at least \(4\), as well as other related graphs, and proposed the problem of constructing more such trees. Kandekar and Mane \cite{kandekar_constructing_2024} constructed \(3\) completely independent spanning trees in hypercubes of dimension at least \(7\).

We show that, for any \(k\), there are \(k\) completely independent spanning trees in \(Q_n\) for sufficiently large \(n\). Our main result is the following:
\begin{theorem}
    There exist \(\Omega(n)\) completely independent spanning trees in \(Q_n\).\label{thm:start}
\end{theorem}
Note that of course there is a linear upper bound on the number of completely independent spanning trees in \(Q_n\)---for example, by considering the number of edges in each spanning tree of \(Q_n\), we see that \(Q_n\) can have at most \(\left(\frac12+o(1)\right)n\) edge-disjoint spanning trees.

We now turn to diameter, which in applications to communications corresponds to the maximal delay in a network \cite{pai_constructing_2016}. Any spanning tree of \(Q_n\) must have diameter at least \(2n-1\). With a more involved construction than for Theorem \ref{thm:start}, we will show the following stronger result:

\begin{theorem}
    There exist \(\left(\frac1{12}+o(1)\right)n\) completely independent spanning trees in \(Q_n\), each with diameter \((2+o(1))n\).
\end{theorem}

Finally, we also consider a more general class of graphs: the Cartesian powers \(H^{n}\), where two vertices \(\mathbf{u}=(u_1,\dots,u_n),\mathbf{u}'\) are adjacent if there is some \(i\in [n]\) such that \(u_j=u'_j\) for each \(j\neq i\), and \(u_iu'_i\in H\).

Again we show that, for a fixed graph \(H\), for any \(k\), there are \(k\) completely independent spanning trees in \(H^n\) for sufficiently large \(n\). We note that any spanning tree must have diameter at least \(2n\cdot r(H)-1\), where \(r(H)\) is the radius of \(H\), and generalise the previous result to the following:
\begin{theorem}
Given any connected graph \(H\), there exist \(\left(\frac1{12}+o(1)\right)n\) completely independent spanning trees in \(H^{n}\), each with diameter \((2+o(1))n\cdot r(H)\).
\end{theorem}

There has been a considerable amount of work on completely independent spanning trees. For example, Araki \cite{araki_diracs_2014} showed that the condition of Dirac's theorem on Hamilton cycles, that a graph's minimum degree be at least half its order, is sufficient for the existence of two completely independent spanning trees. Hasunuma \cite{hasunuma_completely_2002} showed that any \(4\)-connected maximal planar graph has at least two, but that in general deciding whether a graph has two completely independent spanning trees is NP-complete.

Hasunuma also conjectured that any \(2k\)-connected graph must include \(k\) completely independent spanning trees. Péterfalvi \cite{peterfalvi_two_2012} disproved this, constructing graphs of arbitrarily large connectivity without even two completely independent spanning trees. Another counterexample was given by Pai, Yang, Yao, Tang, and Chang \cite{pai_completely_2014}, who showed that \(Q_{10}\) does not have \(5\) completely independent spanning trees. Other notions of independence have been studied on the hypercube---Tang, Wang, and Leu \cite{tang_optimal_2004} found \(n\) trees in \(Q_n\) rooted at the same vertex such that paths between a vertex and the root in different trees do not meet except at their endpoints. Cartesian powers have also been studied: Hong \cite{hong_completely_2018} proved that for any connected graph \(H\) with minimum degree at least \(d\) and order at least \(2d\), \(H^d\) has \(d\) completely independent spanning trees.

For other related results, see Cheng, Wang, and Fan's 2023 survey \cite{cheng_independent_2023} of research on independent spanning trees.

\section{Overview}

In this section we give an overview of some of the ideas used in the proofs, omitting technical details.

We say the \emph{interior} of a tree is the set of its vertices which are not leaves. Notice that the interior of a path are those vertices in the path other than the endpoints. An equivalent formulation of complete independence was noted by Hasunuma \cite{hasunuma_completely_2001}; we give a proof for the sake of completeness.

\begin{lemma}
    A set of spanning trees of a graph are completely independent if and only if they are edge-disjoint, and their interiors are disjoint.\label{lem:characterisation}
\end{lemma}
\begin{proof}
    By definition, spanning trees are completely independent if they are edge-disjoint, and for any two vertices, the paths between them in each tree have disjoint interior.
    But now notice that any vertex in the interior of such a path has degree at least two, so is in the interior of the tree. Hence the interior of a path within a tree is included in the interior of the tree itself.
    So if the trees are edge-disjoint and have disjoint interiors, then they must be completely independent.

    By definition, if a set of spanning trees are not edge-disjoint, then they are not completely independent. Now if a set of spanning trees do not have disjoint interiors, then some vertex \(v\) lies in the interior of two trees \(T_1\) and \(T_2\). Now as \(v\) is not a leaf in either tree, removing \(v\) disconnects either tree. Thus each \(T_i-v\) has more than one component. Now some \(u,w\) must be in different components from each other in each \(T_i-v\). But then the paths in \(T_i\) between these both contain \(v\), and so our spanning trees are not completely independent.
\end{proof}

We will regard \(Q_n\) as
\[Q_m\times Q_{n-m}=\left\{(u;w):u\in Q_m, w\in Q_{n-m}\right\},\]
for some relevant value of \(m\). We will often consider `rows' of the form \(Q_m\times\{w\}\) and `columns' of the form \(\{u\}\times Q_{n-m}\), and with slight abuse of notation talk about \(Q_m\) to mean \(Q_m\times\{w\}\), and so on.

Let \(F\) be a vector space over \(\mathbb{F}_2\) of dimension \(k\). In Section 3, we will construct \(2^k\) completely independent spanning trees \(\left\{T_\alpha:\alpha\in F\right\}\) as follows: we first label the vertices of \(Q_m\) by \(F\) in such a way that each vertex is adjacent to eight vertices with each other possible label, but not to any with the same label. Call this labelling \(\ell:Q_m\to F\), and extend it to \(Q_n\) by \(\ell(u;w)=\ell(u)\), so that it labels entire columns.

We now describe \(T_\alpha\). Most rows \(Q_m\times\{w\}\) will be `standard' rows, which will be identical, and will include in \(T_\alpha\), for each vertex with a label other than \(\alpha\), an edge between that vertex and one with label \(\alpha\). Every column \(\{u\}\times Q_{n-m}\) of \(T_\alpha\) will either contain no edges, or contain a copy of some fixed spanning tree \(T'\) of \(Q_{n-m}\), which will have a large number of leaves. Certain rows \(Q_m\times\{w\}\), where \(w\) is a leaf of \(T'\) will instead be `junction' rows, which will include a path in some \(T_\alpha\) between two vertices of label \(\alpha\), and otherwise behave like standard rows. These junction rows will serve to connect the copies of \(T'\) within \(T_\alpha\).

The interior of each \(T_\alpha\) will then be those vertices labelled by \(\alpha\), together with the paths between copies of \(T'\) at junctions, minus those vertices used in such paths for other trees \(T_\beta\). These will thus be disjoint. By giving ourselves enough dimensions to have each vertex adjacent to not one, but \(8\) other vertices of each possible label, we will have enough leeway to build our junction rows so that the spanning trees remain edge-disjoint. Thus the \(T_\alpha\) will be completely independent.

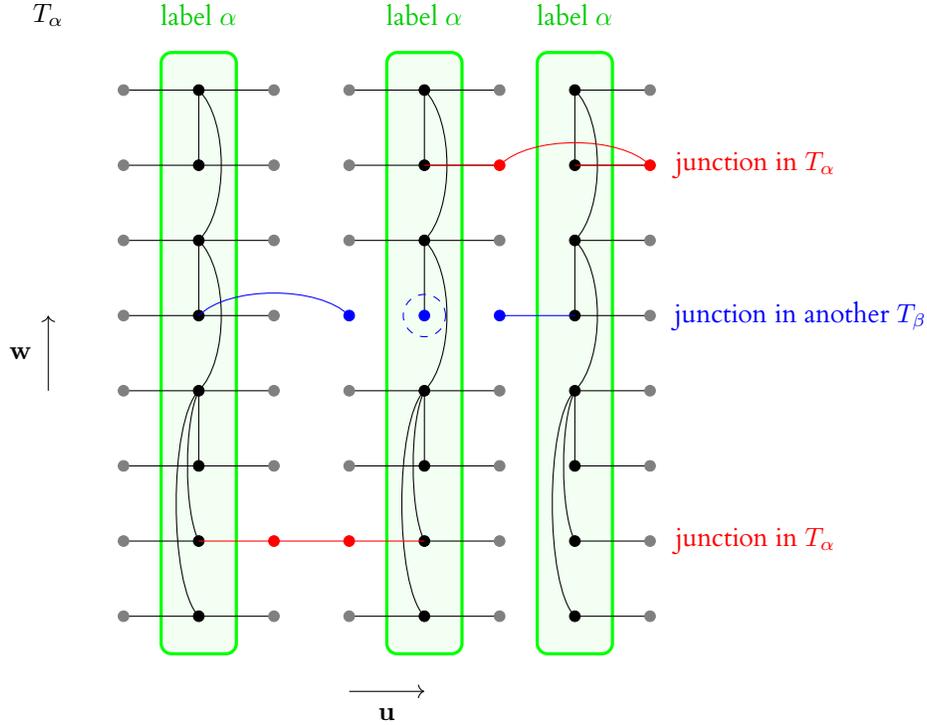
\begin{figure}[h]
    \centering
    \begin{tikzpicture}
    \filldraw[color=green, fill=green!5, very thick, rounded corners] (1.5,-0.5) rectangle (2.5,7.5) {};
    \draw (2,5) .. controls (2.4,5.4) and (2.4,6.6) .. (2,7);
    \draw (2,7) -- (2,6);
    \draw (2,3) .. controls (1.8,2.6) and (1.8,1.4) .. (2,1);
    \draw (2,3) .. controls (1.6,2.6) and (1.6,0.4) .. (2,0);
    \draw (2,3) .. controls (2.4,3.4) and (2.4,4.6) .. (2,5);
    \draw (2,3) -- (2,2);
    \draw (2,5) -- (2,4);
    \filldraw (2,0) circle (2pt);
    \filldraw (2,1) circle (2pt);
    \filldraw (2,2) circle (2pt);
    \filldraw (2,3) circle (2pt);
    \filldraw (2,4) circle (2pt);
    \filldraw (2,5) circle (2pt);
    \filldraw (2,6) circle (2pt);
    \filldraw (2,7) circle (2pt);
    
    \filldraw[color=green, fill=green!5, very thick, rounded corners] (4.5,-0.5) rectangle (5.5,7.5) {};
    \draw (5,5) .. controls (5.4,5.4) and (5.4,6.6) .. (5,7);
    \draw (5,7) -- (5,6);
    \draw (5,3) .. controls (4.8,2.6) and (4.8,1.4) .. (5,1);
    \draw (5,3) .. controls (4.6,2.6) and (4.6,0.4) .. (5,0);
    \draw (5,3) .. controls (5.4,3.4) and (5.4,4.6) .. (5,5);
    \draw (5,3) -- (5,2);
    \draw (5,5) -- (5,4);
    \filldraw (5,0) circle (2pt);
    \filldraw (5,1) circle (2pt);
    \filldraw (5,2) circle (2pt);
    \filldraw (5,3) circle (2pt);
    \filldraw[blue] (5,4) circle (2pt);
    \filldraw (5,5) circle (2pt);
    \filldraw (5,6) circle (2pt);
    \filldraw (5,7) circle (2pt);
    
    \filldraw[color=green, fill=green!5, very thick, rounded corners] (6.5,-0.5) rectangle (7.5,7.5) {};
    \draw (7,5) .. controls (7.4,5.4) and (7.4,6.6) .. (7,7);
    \draw (7,7) -- (7,6);
    \draw (7,3) .. controls (6.8,2.6) and (6.8,1.4) .. (7,1);
    \draw (7,3) .. controls (6.6,2.6) and (6.6,0.4) .. (7,0);
    \draw (7,3) .. controls (7.4,3.4) and (7.4,4.6) .. (7,5);
    \draw (7,3) -- (7,2);
    \draw (7,5) -- (7,4);
    \filldraw (7,0) circle (2pt);
    \filldraw (7,1) circle (2pt);
    \filldraw (7,2) circle (2pt);
    \filldraw (7,3) circle (2pt);
    \filldraw (7,4) circle (2pt);
    \filldraw (7,5) circle (2pt);
    \filldraw (7,6) circle (2pt);
    \filldraw (7,7) circle (2pt);

    \draw (1,0) -- (3,0);
    \draw (1,1) -- (2,1);
    \draw (1,2) -- (3,2);
    \draw (1,3) -- (3,3);
    \draw (1,4) -- (3,4);
    \draw (1,5) -- (3,5);
    \draw (1,6) -- (3,6);
    \draw (1,7) -- (3,7);
    
    \filldraw[gray] (1,0) circle (2pt);
    \filldraw[gray] (1,1) circle (2pt);
    \filldraw[gray] (1,2) circle (2pt);
    \filldraw[gray] (1,3) circle (2pt);
    \filldraw[gray] (1,4) circle (2pt);
    \filldraw[gray] (1,5) circle (2pt);
    \filldraw[gray] (1,6) circle (2pt);
    \filldraw[gray] (1,7) circle (2pt);
    
    \filldraw[gray] (3,0) circle (2pt);
    \filldraw[red] (3,1) circle (2pt);
    \filldraw[gray] (3,2) circle (2pt);
    \filldraw[gray] (3,3) circle (2pt);
    \filldraw[gray] (3,4) circle (2pt);
    \filldraw[gray] (3,5) circle (2pt);
    \filldraw[gray] (3,6) circle (2pt);
    \filldraw[gray] (3,7) circle (2pt);

    \draw (4,0) -- (6,0);
    \draw (5,1) -- (6,1);
    \draw (4,2) -- (6,2);
    \draw (4,3) -- (6,3);
    \draw (4,5) -- (6,5);
    \draw (4,6) -- (6,6);
    \draw (4,7) -- (6,7);
    
    \filldraw[gray] (4,0) circle (2pt);
    \filldraw[red] (4,1) circle (2pt);
    \filldraw[gray] (4,2) circle (2pt);
    \filldraw[gray] (4,3) circle (2pt);
    \filldraw[blue] (4,4) circle (2pt);
    \filldraw[gray] (4,5) circle (2pt);
    \filldraw[gray] (4,6) circle (2pt);
    \filldraw[gray] (4,7) circle (2pt);
    
    \filldraw[gray] (6,0) circle (2pt);
    \filldraw[gray] (6,1) circle (2pt);
    \filldraw[gray] (6,2) circle (2pt);
    \filldraw[gray] (6,3) circle (2pt);
    \filldraw[blue] (6,4) circle (2pt);
    \filldraw[gray] (6,5) circle (2pt);
    \filldraw[red] (6,6) circle (2pt);
    \filldraw[gray] (6,7) circle (2pt);

    \draw (7,0) -- (8,0);
    \draw (7,1) -- (8,1);
    \draw (7,2) -- (8,2);
    \draw (7,3) -- (8,3);
    \draw (7,4) -- (8,4);
    \draw (7,5) -- (8,5);
    \draw (7,6) -- (8,6);
    \draw (7,7) -- (8,7);
    
    \filldraw[gray] (8,0) circle (2pt);
    \filldraw[gray] (8,1) circle (2pt);
    \filldraw[gray] (8,2) circle (2pt);
    \filldraw[gray] (8,3) circle (2pt);
    \filldraw[gray] (8,4) circle (2pt);
    \filldraw[gray] (8,5) circle (2pt);
    \filldraw[red] (8,6) circle (2pt);
    \filldraw[gray] (8,7) circle (2pt);

    \draw[red] (2,1) -- (5,1);
    \draw[red] (5,6) -- (6,6);
    \draw[red] (7,6) -- (8,6);
    \draw[red] (6,6) .. controls (6.4,6.4) and (7.6,6.4) .. (8,6);

    \draw[blue] (6,4) -- (7,4);
    \draw[blue] (4,4) .. controls (3.6,4.4) and (2.4,4.4) .. (2,4);
    \draw[blue, dashed] (5,4) circle (8pt);

    \draw[->] (4,-1) -- node[below=1mm] {\(\mathbf{u}\)} (5,-1);
    \draw[->] (0,3) -- node[left=1mm] {\(\mathbf{w}\)} (0,4);

    \draw (0,8) node {\(T_\alpha\)};

    \draw (2,8) node [green!80!black] {label \(\alpha\)};
    \draw (5,8) node [green!80!black] {label \(\alpha\)};
    \draw (7,8) node [green!80!black] {label \(\alpha\)};
    \draw (9.4,6) node [red] {junction in \(T_\alpha\)};
    \draw (10,4) node [blue] {junction in another \(T_\beta\)};
    \draw (9.4,1) node [red] {junction in \(T_\alpha\)};
    \end{tikzpicture}
    \caption{The structure of tree \(T_\alpha\). Here vertices of \(Q_n\) are represented as \((u;w)\in Q_m\times Q_{n-m}\). Vertices of label \(\alpha\) are shown in green, with junctions of \(T_\alpha\) indicated in red, and those of another \(T_\beta\) in blue.}
\end{figure}

In Section 4, we will evaluate the diameter of the spanning trees constructed in the previous section, and find that it is quadratic in the number of trees, and so in \(n\). In Section 5, we will adapt our construction to fit more useful edges into each junction row, at the cost of a slightly more technical proof. This will allow us to use a finite number of junction rows to connect each \(T_\alpha\), and in particular allow us to reduce the diameter of each \(T_\alpha\) to \((2+o(1))n\), which is the best possible constant.

In Section 6, we extend our construction from spanning trees in hypercubes to arbitrary Cartesian powers \(H^{n}\). We note that spanning trees of \(H^{n}\) have diameter at least \(2n\cdot r(H) - 1\). We show that we can construct the same number of completely independent spanning trees as we could in \(Q_n\), while ensuring they all have diameter \((2+o(1))n\cdot r(H)\).

\section{Linearly many spanning trees in $Q_n$}
In this section, we prove the following result:

\begin{theorem}
    For \(k\) a positive integer, and \(n\geq 16\cdot 2^{k} -1\), there exist \(2^k\) completely independent spanning trees in \(Q_n\).\label{thm:firsttheorem}
\end{theorem}

We consider \(Q_n\) to have vertex set \(\mathbb{F}_2^n\), with standard basis \(\{\mathbf{e}_i: i\in [n]\}\). Then the edges of \(Q_n\) are 
\[E\left(Q_n\right)=\left\{\left\{\mathbf{u},\mathbf{u}+\mathbf{e}_i\right\}:\mathbf{u}\in \mathbb{F}_2^n,i\in [n]\right\}.\]
Notice that equivalently, \(\mathbf{u}\mathbf{u}'\) is an edge of \(Q_n\) if and only if \(\mathbf{u}+\mathbf{u}'=\mathbf{e}_i\) for some \(i\) which we call the \emph{direction} of \(\mathbf{u}\mathbf{u}'\). We set \(F=\mathbb{F}_2^k\), and construct \(2^k\) completely independent spanning trees \(\left\{T_\alpha:\alpha\in F\right\}\) of \(Q_n\) for any \(n\geq2m+1\), where \(m=8\left(2^k-1\right)\). We will regard \(Q_n\) as \(Q_m\times Q_{n-m}\), and notice that every edge is either contained in a \emph{row} \(Q_m\times\{\mathbf{w}\}\), or a \emph{column} \(\{\mathbf{u}\}\times Q_{n-m}\). We will therefore define the behaviour of our trees within rows and columns separately.

\subsection{The labelling}

We first construct our labelling \(\ell:Q_m\to F\) by identifying the \(m\) coordinates of \(Q_m\) with \([8]\times F\setminus \{0\}\). Thus the vertex \(\mathbf{u}\in Q_m\) has coordinates \(u_{(i,\alpha)}\in\mathbb{F}_2\), and we may refer to the basis vectors \(\mathbf{e}_{(i,\alpha)}\). We label \(\mathbf{u}\) with
\[\ell(\mathbf{u})=\sum_{(i,\alpha)\in[8]\times F\setminus\{0\}}u_{(i,\alpha)}\alpha.\]
Now for any \(\mathbf{u},i,\alpha\), we find that \(\ell\left(\mathbf{u}+\mathbf{e}_{(i,\alpha)}\right)=\ell(\mathbf{u})+\alpha\). In particular, any vertex \(\mathbf{u}\) has \(8\) neighbours with label \(\alpha\), namely
\[\{\mathbf{u}+\mathbf{e}_{(i,\ell(\mathbf{u})+\alpha)}: i\in [8]\}.\]

We will regard the column \(\{\mathbf{u}\}\times Q_{n-m}\) as all having the label \(\ell(\mathbf{u})\).

It is interesting, although not relevant, to notice that this labelling is essentially code-theoretic. If we had identified the dimensions of \(Q_m\) with a single copy of \(F\setminus{0}\), then this labelling would construct a Hamming code \(\ell^{-1}(0)\). In our actual labelling, each set \(\ell^{-1}(\alpha)\) corresponds a code in which each block represents \(2^k-1\) bytes, whose parities are constrained to be in a Hamming code---so if one error is transmitted, this code would allow the byte in which it occurred to be identified.

\subsection{The columns}

We define a spanning tree \(T'\) of \(Q_{n-m}\) with a large number of leaves. Set \[S_0=\{\mathbf{u}\in Q_{n-m}: u_{n-m}=0\},\] and notice that \(Q_{n-m}[S_0]\) is isomorphic to \(Q_{n-m-1}\), and so is connected. Take \(T''\) to be a spanning tree of \(Q_{n-m}[S_0]\). Define
\[T'=T''\cup \left\{\{\mathbf{u},\mathbf{u}+\mathbf{e}_{n-m}\}:\mathbf{u}\in S_0\right\}.\]
This joins a new leaf to each vertex of \(T''\), so gives a spanning tree of \(Q_{n-m}\) with at least \(2^{n-m-1}\) leaves. Now on each column \(\{\mathbf{u}\}\times Q_{n-m}\), the tree \(T_\alpha\) will have induced subgraph
\begin{align*}T_\alpha\left[\{\mathbf{u}\}\times Q_{n-m}\right]={}&{}T'\text{ if }\ell(\mathbf{u})=\alpha,\text{ and}\\{}&{}\varnothing\text{ otherwise.}\end{align*}
Note that here we have abusively identified the vertices of \(Q_{n-m}\) with \(\{\mathbf{u}\}\times Q_{n-m}\). Thus in tree \(T_\alpha\), for any vertex \(\mathbf{u}\) of \(Q_m\) with label \(\alpha\), column \(\{\mathbf{u}\}\times Q_{n-m}\) is connected. Additionally, notice that no edge within a column is in more than one \(T_\alpha\), as each column \(\{\mathbf{u}\}\times Q_{n-m}\) only contains edges of \(T_{\ell(\mathbf{u})}\). 

\subsection{The structure}

We would now like to add paths within each \(T_\alpha\) that connect our copies of \(T'\) between columns. We will do this in carefully designed `junction' rows. We first observe that it is sufficient to use paths of length \(3\).

\begin{claim}
    The graph \(G_\alpha\) with vertex set \(\{\mathbf{u}\in Q_m:\ell(\mathbf{u})=\alpha\}\) in which two vertices are adjacent if there is path of length \(3\) between them is connected.\label{clm:Galphaconn}
\end{claim}
\begin{proof}
    We know that \(Q_m\) is connected, so for any \(\mathbf{u},\mathbf{u}'\) in \(Q_m\), there is a path \(\mathbf{u}=\mathbf{u}^{(0)},\mathbf{u}^{(1)},\dots,\mathbf{u}^{(r)}=\mathbf{u}'\) between them in \(Q_m\). The constraint \(\ell(\mathbf{u})=\ell(\mathbf{u}')\) implies that \(r>1\), as otherwise some \(\mathbf{e}_{(i,\alpha)}=\mathbf{u}+\mathbf{u}'\). For the case \(r=2\), choose \(\mathbf{u}^{(2)}\) a neighbour of \(\mathbf{u}^{(1)}\) of label other than \(\alpha\), and \(u_3\) a neighbour of \(\mathbf{u}^{(2)}\) with label \(\alpha\). So in \(G_\alpha\), \(\mathbf{u}\sim \mathbf{u}^{(3)} \sim \mathbf{u}'\).

    If \(r=3\) we're trivially done. Equally, if any \(\mathbf{u}_i\) in the interior of the path has label \(\alpha\), we're done by induction on the length of the path. So we just need to deal with the case where \(r\geq 4\) and \(\ell(\mathbf{u}^{(i)})\neq\alpha\) for all \(2\leq i \leq r=2\). But now define \(\mathbf{u}^{(i)*}\) for such \(i\) to be a neighbour of \(\mathbf{u}^{(i)}\) with label \(\alpha\), and notice that in \(G_\alpha\), \[\mathbf{u}\sim \mathbf{u}^{(2)*}\sim \mathbf{u}^{(3)*}\sim \cdots \sim \mathbf{u}^{(r-2)*}\sim \mathbf{u}'.\] Thus \(G_\alpha\) is connected.
\end{proof}

In particular, for each \(G_\alpha\), we can choose a spanning tree \(K_\alpha\). Each \(K_\alpha\) has fewer edges than vertices, and the vertex sets \(V(K_\alpha)\) partition \(V(Q_m)\), so there are fewer than \(2^m\) pairs \((\alpha,e)\) for \(e\in E(K_\alpha)\). But \(T'\) has at least \(2^{n-m-1}\) leaves, and \(n\geq 2m+1\) ensures that this is at least \(2^m\). Thus for each \(\alpha\) and edge \(\mathbf{u}\mathbf{u}'\in K_\alpha\), we can choose a distinct leaf \(\mathbf{w}_{\alpha,\mathbf{u}\mathbf{u}'}\) of \(T'\).

We will select edges to be in \(T_\alpha\) within row \(Q_m\times\mathbf{w}_{\alpha,\mathbf{u}\mathbf{u}'}\) that include a path of length \(3\) between \(\mathbf{u}\) and \(\mathbf{u}'\). This will ensure that all vertices of label \(\alpha\) are connected in \(T_\alpha\). We call these `junction' rows, and other rows `standard' rows. We will also ensure that every other vertex is a leaf of \(T_\alpha\), adjacent to a vertex of label \(\alpha\). This will build a spanning tree, and with care we can ensure that the \(T_\alpha\) are completely independent using the characterisation of Lemma \ref{lem:characterisation}.

\subsection{Standard rows}
Impose some linear order \(<\) on the elements of \(F\). Now for any \(\mathbf{w}\) which is not of the form \(\mathbf{w}_{\alpha,\mathbf{u}\mathbf{u}'}\) for any \(\alpha\in F\) and \(\mathbf{u}\mathbf{u}'\in K_\alpha\), we call \(Q_m\times \{\mathbf{w}\}\) a \emph{standard row}, in which \(T_\alpha\) has edges
\begin{align*}
E\left(T_\alpha\left[Q_m\times \{\mathbf{w}\}\right]\right)={}&{}\left\{\left\{\mathbf{u},\mathbf{u}+\mathbf{e}_{(1,\alpha+\ell(\mathbf{u}))}\right\}:\mathbf{u}\in Q_m,\ell(\mathbf{u})< \alpha\right\}\cup\\{}&{}\left\{\left\{\mathbf{u},\mathbf{u}+\mathbf{e}_{(2,\alpha+\ell(\mathbf{u}))}\right\}:\mathbf{u}\in Q_m,\alpha<\ell(\mathbf{u})\right\}.
\end{align*}

Note that each of these edges is between some \(\mathbf{u}\) of label other than \(\alpha\), and a vertex of label \(\alpha\), and for each such \(\mathbf{u}\) there is exactly one such edge. As \(T_\alpha\) contains no edges from column \(\{\mathbf{u}\}\times Q_{n-m}\), we find that \(\mathbf{u}\) is a leaf of \(T_\alpha\). Hence each vertex \(\mathbf{u}\) can only be in the interior of \(T_{\ell(\mathbf{u})}\). This is one of the conditions for Lemma \ref{lem:characterisation} that we will need. We now check the other:

\begin{claim}
    No edge of a standard row is in more than one \(T_\alpha\).\label{claim:obvious}
\end{claim}
\begin{proof}
    Let \(\mathbf{u}\mathbf{u}'\) be an edge of standard row \(Q_m\times \{\mathbf{w}\}\). Write \(\alpha=\ell(\mathbf{u})\) and \(\alpha'=\ell(\mathbf{u}')\), and suppose \(\alpha<\beta\) w.l.o.g. Now by the definition of a standard row, \(\mathbf{u}\mathbf{u}'\) is in \(T_{\alpha'}\) if and only if it has direction \((1,\alpha+\alpha')\), and \(T_\alpha\) if and only if it has direction \((2,\alpha+\alpha')\). Hence it cannot be in both of these, so as it cannot be in any other \(T_\beta\) by definition, we have the desired.
\end{proof}

\subsection{Junction rows}
We now define the junction rows \(Q_m\times \{\mathbf{w}_{\alpha,\mathbf{a}\mathbf{b}}\}\). Choose a path \(P\) on vertices \[\mathbf{a}=\mathbf{u}^{(0)}\sim\mathbf{u}^{(1)}\sim\mathbf{u}^{(2)}\sim\mathbf{u}^{(3)}=\mathbf{b}\] in \(Q_m\), and w.l.o.g.\ take the \(3\) edges of this path to be in \([3]\times F\setminus\{0\}\). We first define \(T_\alpha\)'s edges by
\[E\left(T_\alpha\left[Q_m\times \{\mathbf{w}_{\alpha,\mathbf{u}\mathbf{u}'}\}\right]\right)=P\cup\left\{\left\{\mathbf{u},\mathbf{u}+\mathbf{e}_{(4,\alpha+\ell(\mathbf{u}))}\right\}:\mathbf{u}\in Q_m\setminus\{\mathbf{u}^{(1)},\mathbf{u}^{(2)}\},\ell(\mathbf{u})\neq \alpha\right\}.\]
Thus all vertices of \(Q_m\setminus\{\mathbf{u}^{(1)},\mathbf{u}^{(2)}\}\) with label other than \(\alpha\) are leaves of \(T_\alpha\). We now modify our design for standard rows to ensure that \(\mathbf{u}^{(1)}\) and \(\mathbf{u}^{(2)}\) are both leaves of all other \(T_\beta\). For each choice of \(\beta\in F\setminus\{\alpha,\ell(\mathbf{u}^{(1)}),\ell(\mathbf{u}^{(2)}\}\), we simply take
\begin{align*}
E\left(T_\beta\left[Q_m\times \{\mathbf{w}\}\right]\right)={}&{}\left\{\left\{\mathbf{u},\mathbf{u}+\mathbf{e}_{(5,\beta+\ell(\mathbf{u}))}\right\}:\mathbf{u}\in Q_m,\ell(\mathbf{u})< \beta\right\}\cup\\{}&{}\left\{\left\{\mathbf{u},\mathbf{u}+\mathbf{e}_{(6,\beta+\ell(\mathbf{u}))}\right\}:\mathbf{u}\in Q_m,\beta<\ell(\mathbf{u})\right\}.
\end{align*}
When building \(T_{\ell(\mathbf{u}^{(i)})}\) for \(i=1,2\), we also have to make sure that \(\mathbf{u}^{(i)}\) is a leaf in order to meet the condition of Lemma \ref{lem:characterisation}. One edge of this tree in the column \(\{\mathbf{u}^{(i)}\}\times Q_{n-m}\) is already incident to \(\mathbf{u}^{(i)}\), so there cannot be any others.  We will build the same construction as other \(T_\beta\) but replace any edges incident to \(\mathbf{u}^{(i)}\). Thus for \(i=1,2\) we take

\begin{align*}
S_i={}&{}\left\{\left\{\mathbf{u},\mathbf{u}+\mathbf{e}_{(5,{\ell(\mathbf{u}^{(i)})}+\ell(\mathbf{u}))}\right\}:\mathbf{u}\in Q_m,\ell(\mathbf{u})< {\ell(\mathbf{u}^{(i)})}\right\}\cup\\{}&{}\left\{\left\{\mathbf{u},\mathbf{u}+\mathbf{e}_{(6,{\ell(\mathbf{u}^{(i)})}+\ell(\mathbf{u}))}\right\}:\mathbf{u}\in Q_m,{\ell(\mathbf{u}^{(i)})}<\ell(\mathbf{u})\right\}.
\end{align*}
Note this matches \(T_\beta\) for other \(\beta\). Finally define
\[E\left(T_{\ell(\mathbf{u}^{(i)})}\left[Q_m\times \{\mathbf{w}\}\right]\right)=\left\{e:e\in S_i,\mathbf{u}^{(i)}\notin e\right\}\cup\left\{\left\{\mathbf{u},\mathbf{u}+\mathbf{e}_{(6+i,{\ell(\mathbf{u}^{(i)})}+\ell(\mathbf{u}))}\right\}:\mathbf{u}\mathbf{u}^{(i)}\in S_i\right\}.\]
Thus in this row, every vertex is a leaf of all but one \(T_\beta\), and by the reasoning of Claim \ref{claim:obvious}, no edge is in two \(T_\beta\). And each vertex of label other than \(\beta\) is adjacent to one.

Now we have constructed our trees \(T_\alpha\). Note that within each \(T_\alpha\), all vertices of label other than \(\a lpha\) are adjacent to one of label \(\alpha\). But all of these are connected by the copies of \(T'\) in their columns, and paths of each junction row \(Q_m\times \{\mathbf{w}_{\beta,\mathbf{a}\mathbf{b}}\}\). Thus these are indeed spanning trees. But by construction, these trees are edge-disjoint, and no vertex is in the interior of more than one \(T_\alpha\). Thus they are completely independent, and so we have proven Theorem \ref{thm:firsttheorem}.

\section{Evaluating the diameters}
We now ask how small we can make the diameters of our spanning trees. So far, we have not tried to control \(T'\) and the \(G_\alpha\) in a way to minimise this. We define the \emph{broadcast tree} on \(Q_n\) to have edges
\[E=\left\{\left\{\mathbf{u},\mathbf{u}+\mathbf{e}_i\right\}: \mathbf{u}\in Q_n, u_j=0\ \forall j\geq i\right\}.\]

\begin{claim}
    The broadcast tree on \(Q_n\) has diameter \(2n-1\), and this is minimal across all spanning trees.\label{clm:minimalconstant}
\end{claim}
\begin{proof}
    Within the broadcast tree, each vertex \(\mathbf{u}\) has distance \(\left|\left\{i:u_i=1\right\}\right|\) from \(\mathbf{0}\). But the only vertex with distance \(n\) from \(\mathbf{0}\) is \(\mathbf{1}=(1,\dots,1)\), so it has diameter \(2n-1\).

    Likewise, any spanning tree has a centre which is a vertex, or pair of vertices, whose maximum distance to any other vertex is minimal. This is the middle, or middle two vertices, of any longest path, so it suffices to show a vertex in the centre has distance \(n\) from some other vertex. But any vertex \(\mathbf{u}\) has distance \(n\) from \(\mathbf{1}-\mathbf{u}\).
\end{proof}

We can take the broadcast tree on \(Q_{n-m}\) to be our \(T'\), as it has the required number of leaves. We will now show that we can use the same structure on \(G_\alpha\).

\begin{claim}
    Each \(G_\alpha\) contains a copy of \(Q_{m-k}\) with  the same vertex set as \(G_\alpha\).
\end{claim}
\begin{proof}
    Notice that by Claim \ref{clm:Galphaconn}, the vertices of \(G_\alpha\) are generated as a vector space by
    \[S_\alpha=\left\{\mathbf{e}_i+\mathbf{e}_j+\mathbf{e}_k:\ell\left(\mathbf{e}_i+\mathbf{e}_j+\mathbf{e}_k\right)=\alpha\right\}.\]
    Let \(B_\alpha\subset S_\alpha\) be a basis of \(V\left(G_\alpha\right)\). Notice that for any choice of \(\mathbf{u}\in G_\alpha,\mathbf{b}\in B_\alpha\), we have \(\mathbf{u}(\mathbf{u}+\mathbf{b})\in G_\alpha\). Hence \(G_\alpha\) contains a copy of \(Q_{|B_\alpha|}\) with the same vertex set. But
    \[\left|V\left(G_\alpha\right)\right|=2^{m-k},\]
    so we must have \(\left|B_\alpha\right|=m-k\). 
\end{proof}

This guarantees that we can choose our spanning trees \(K_\alpha\) to have diameter at most \(2(m-k)-1\). Notice that each edge of \(K_\alpha\) corresponds to a path of length \(3\) in \(Q_n\). A path between two vertices in \(T_\alpha\) must use the junction rows of a corresponding path in \(K_\alpha\), and travel between these junction rows using the copies of \(T'\) in the columns.

Pick distinct \(\mathbf{a}_{\alpha}\in Q_{n-2m-1+k}\), noting that we can do this as \(n\geq 2m-1\). Now each \(\{\mathbf{a}_{\alpha}\}\times Q_{m-k+1}\in Q_{n-m}\) induces a broadcast tree as a subgraph of \(T'\). By choosing the \(\mathbf{w}_{\alpha,e}\) from this, we can guarantee the paths between junction rows are no longer than \(2(m-k+1)-1\)

If the endpoints are not of label \(\alpha\), we will also need an edge within a row at either end of the path. This shows the diameter of \(T_\alpha\) is at most
\[2+2(2(n-m)-1)+3\left(2(m-k)-1\right)+(2(m-k)-2)(2(m-k+1)-1)=4n+\Theta\left(2^{2k}\right).\]
This is quadratic in the number of spanning trees, and so quadratic in \(n\) when we build linearly many spanning trees. In fact we will be able to construct \(T_\alpha\) with diameter \((2+o(1))n\) with a more careful construction.

\section{Constructing trees with linear diameter}
Our previous construction was inefficient in that it used a whole row for each individual edge of \(K_\alpha\). This meant that paths had to travel between rows a large number of times. We redesign our row structure to use a constant number of junction rows. We will be able to construct linearly many completely independent spanning trees with linear diameter. Indeed, our constant for the diameter is the best possible, by Claim \ref{clm:minimalconstant}.

\begin{theorem}
    For \(k\) a positive integer, and \(n\geq 6\cdot 2^k + k + 4\), there exist \(2^k\) completely independent spanning trees in \(Q_n\) of diameter at most \(2(n+9k+11)=\left(2+o(1)\right)n\).\label{thm:secondtheorem}
\end{theorem}

We again consider \(Q_n=Q_m\times Q_{n-m}\), but now take \(m=6\left(2^k-1\right)\), and identify the dimensions of \(Q_m\) with \([6]\times F\setminus\{0\}\) as before. We also have \(n-m\geq k+10\). As before, we label \(Q_m\) by
\[\ell(\mathbf{u})=\sum_{(i,\alpha)\in[6]\times F\setminus\{0\}}u_{(i,\alpha)}\alpha.\]
\subsection{The columns}\label{subsec:lindiam-columns}
We define the edges of each \(T_\alpha\) within columns in the same way, except we now define \(T'\) more carefully.
Let \(T''\) be the broadcast tree on \(Q_{n-m-1}=\{\mathbf{w}\in Q_{n-m}:w_{n-m}=0\}\). Recall that \(F=\mathbb{F}_2^k\), so we can consider \(V(Q_{n-m})=V(Q_{n-m-k-1})\times F \times \mathbb{F}_2\). Now set
\[W_\alpha=\left\{(\mathbf{0},\alpha,1)+\mathbf{e}_i:i \in [n-m-k-1]\right\},\]
and define
\begin{align*}E(T')={}&{}E(T'')\cup\left\{\{\mathbf{w},\mathbf{w}+\mathbf{e}_{n-m}\}:\mathbf{w}\in \left(Q_{n-m}\times\{1\}\right)\setminus W_\alpha\right\}\\{}&{}\cup\left\{\{(\mathbf{0},\alpha,1),\mathbf{w}\}:\mathbf{w}\in W_\alpha\right\}.\end{align*}
Notice that \(T'\) has the same diameter as the broadcast tree on \(Q_{n-m}\), and each vertex still has distance from \(\mathbf{0}\) within \(T'\) at most \(n-m\). In particular, every vertex has distance at most \(n-m+k+1\) from each \((\mathbf{0},\alpha,1)\). Additionally, every \((\mathbf{0},\alpha,1)\) is adjacent to at least \(9\) leaves, namely the vertices of \(W_\alpha\).

\subsection{The rows}
For each \(\alpha\), we will choose \(9\) distinct \(\mathbf{w}_{\alpha,i}\in W_\alpha\), and build junction rows in these. Note that we can do this as we have taken \(n-m-k-1\geq 9\). We will keep the same definition of the standard row, i.e. for all \(\mathbf{w}\) other than these \(\mathbf{w}_{\alpha,i}\), we have
\begin{align*}
E\left(T_\alpha\left[Q_m\times \{\mathbf{w}\}\right]\right)={}&{}\left\{\left\{\mathbf{u},\mathbf{u}+\mathbf{e}_{(1,\alpha+\ell(\mathbf{u}))}\right\}:\mathbf{u}\in Q_m,\ell(\mathbf{u})< \alpha\right\}\cup\\{}&{}\left\{\left\{\mathbf{u},\mathbf{u}+\mathbf{e}_{(2,\alpha+\ell(\mathbf{u}))}\right\}:\mathbf{u}\in Q_m,\alpha<\ell(\mathbf{u})\right\}.
\end{align*}

We first make a general claim that will allow us to build junction rows while meeting the conditions for Lemma \ref{lem:characterisation}.

\begin{claim}
    Suppose \(G\subset Q_m\) meets the following conditions:
    \begin{enumerate}
        \item \(G\) is a forest.
        \item For some \(c\), all the edges of \(G\) have a direction in \(\{1,c\}\times F\setminus \{0\}\).
        \item Every vertex with positive degree in \(G\) is connected to a vertex of label \(\alpha\).
        \item There exist \(r,s\in [6]\setminus\{1,c\}\) and \(a,b\in \mathbb{F}_2\) such that for any vertex \(\mathbf{u}\) of positive degree in \(G\),
        \[\sum_{\alpha\in F\setminus\{0\}}u_{r,\alpha}=a,\]
        and
        \[\sum_{\alpha\in F\setminus\{0\}}u_{s,\alpha}=b.\]
    \end{enumerate}
    Then at a leaf \(\mathbf{w}\) of \(T'\), we can construct a row \(Q_m\times\{\mathbf{w}\}\) such that 
    \begin{enumerate}
        \item \(G\subset T_\alpha\left[Q_m\times\{\mathbf{w}\}\right]\),
        \item For each \(\beta\), the induced subgraph \(T_\beta\left[Q_m\times\{\mathbf{w}\}\right]\) is a forest in which every vertex of \(\mathbf{u}\) is connected to some vertex of label \(\beta\),
        \item the forests \(T_\beta\left[Q_m\times\{\mathbf{w}\}\right]\) are edge-disjoint, and
        \item each vertex \(\mathbf{u}\) is in the interior of at most one resulting \(T_\beta\) on the entire \(Q_n\). 
    \end{enumerate}\label{clm:builder}
\end{claim}
\begin{proof}
    We adapt our previous construction. Without loss of generality, assume that \(c=2,r=5\) and \(s=6\). Let \(V_0\) be the set of vertices of degree zero in \(G\) with labels other than \(\alpha\). Now set
    \[E(T_\alpha\left[Q_m\times\{\mathbf{w}\}\right])=E(G)\cup\left\{\left\{\mathbf{u},\mathbf{u}+\mathbf{e}_{(1,\alpha+\ell(\mathbf{u}))}\right\}:\mathbf{u}\in V_0\right\}.\]
    By the definition of \(V_0\), none of the edges we add form a cycle, as we only add one edge to each isolated vertex of label other than \(\alpha\). So this is indeed a forest in which every vertex of \(\mathbf{u}\) is connected to some vertex of label \(\alpha\).
    
    We now construct the remaining \(T_\beta\left[Q_m\times\{\mathbf{w}\}\right]\), which will be edge-disjoint forests in which every vertex of label other than \(\beta\) is adjacent to a vertex of label \(\beta\) and no other vertices. We must additionally ensure that for each \(\beta\neq\alpha\), any vertex of label \(\beta\) with positive degree in \(G\) has degree zero in \(T_\beta\left[Q_m\times\{\mathbf{w}\}\right]\). This will be enough to meet the conditions of our lemma.

    We will be able to do this using the parity condition we took on \(G\). For each \(\mathbf{u}\) with label other than \(\beta\), we will choose its one neighbour to be
    \begin{align*}
        n_{\beta}(\mathbf{u})={}&{}\mathbf{u}+\mathbf{e}_{(3,\beta+\ell(\mathbf{u}))}{}&{}\text{ if }\ell(\mathbf{u})<\beta\text{ and }\sum_{\alpha\in F\setminus\{0\}}u_{(5,\alpha)}\neq a,\\
        {}&{}\mathbf{u}+\mathbf{e}_{(4,\beta+\ell(\mathbf{u}))}{}&{}\text{ if }\beta<\ell(\mathbf{u})\text{ and }\sum_{\alpha\in F\setminus\{0\}}u_{(6,\alpha)}\neq b,\\
        {}&{}\mathbf{u}+\mathbf{e}_{(5,\beta+\ell(\mathbf{u}))}{}&{}\text{ if }\ell(\mathbf{u})<\beta\text{ and }\sum_{\alpha\in F\setminus\{0\}}u_{(5,\alpha)}=a,\\
        {}&{}\mathbf{u}+\mathbf{e}_{(6,\beta+\ell(\mathbf{u}))}{}&{}\text{ if }\beta<\ell(\mathbf{u})\text{ and }\sum_{\alpha\in F\setminus\{0\}}u_{(6,\alpha)}=b.\\
    \end{align*}

    Now notice that for each \(\mathbf{u}\), the vertex \(n_{\beta}(\mathbf{u})\) has degree zero in \(G\), as by construction, we must have
    \[\left(\sum_{\alpha\in F\setminus\{0\}}u_{(5,\alpha)},\sum_{\alpha\in F\setminus\{0\}}u_{(6,\alpha)}\right)\neq (a,b).\]
    Hence for each \(\beta\), we take
    \[E(T_\beta\left[Q_m\times\{\mathbf{w}\}\right])=\left\{\mathbf{u}n_{\beta}(\mathbf{u}):\mathbf{u}\in Q_m, \ell(\mathbf{u})\neq \beta \right\}.\]
    Now every edge of \(T_\beta\left[Q_m\times\{\mathbf{w}\}\right]\) has an endpoint of degree one, so this is certainly a forest. As before, the ordering ensures that all the \(T_\beta\left[Q_m\times\{\mathbf{w}\}\right]\) are edge-disjoint. Thus we have met all our desired conditions.
\end{proof}

Take \(B\) to be a basis of \(F\), and notice that for every \(\mathbf{u}\in Q_m\), there is exactly one vertex \(\mathbf{u}'\in Q_m\) with label \(\alpha\) such that \(u_{i,\beta}=u'_{i,\beta}\) for every index \((i,\beta)\notin \{1\}\times B\). For \(S\subseteq [6]\times F\setminus\{0\}\), write
\[Q_S=\{\mathbf{u}\in Q_m: u_{(i,\alpha)}=0\ \forall (i,\alpha)\in\left([6]\times F\setminus\{0\}\right)\setminus S\}.\]
The idea will be that junction row \(Q_m\times\{\mathbf{w}_{\alpha,1}\}\) acts like a broadcast tree on \(Q_{[2]\times F\setminus \{0\}}\), and rows \(Q_m\times\{\mathbf{w}_{\alpha,2}\}\) and \(Q_m\times\{\mathbf{w}_{\alpha,3}\}\) will respectively extend this to \(Q_{[3]\times F\setminus \{0\}}\) and \(Q_{[4]\times F\setminus \{0\}}\). We will split the edges that extend this to \(Q_{[5]\times F\setminus \{0\}}\) across the two rows \(Q_m\times\{\mathbf{w}_{\alpha,4},\mathbf{w}_{\alpha,5}\}\), and finally use the remaining \(\mathbf{w}_{\alpha,i}\) to extend this to all the columns of \(Q_m\) with label \(\alpha\).

Let \(\BT(Q_S)\) be the broadcast tree on \(Q_S\), where \(Q_S\) is identified with \(Q_{|S|}\) in the natural way. Hence in \(\BT(Q_S)\), all vertices have distance at most \(|S|\) from \(\mathbf{0}\), and distance at most \(|S|+k\) from any vertex of \(Q_{S\cap\{1\}\times B}\). For \(i,j\in [6]\), \(a,b\in \mathbb{F}_2\), and \(T\subseteq \left([6]\times F\setminus\{0\}\right)\), write
\[\BT(Q_S,T)=\bigcup\left\{\BT(Q_S)+\mathbf{u}:{\mathbf{u}\in Q_{T\setminus S}}\right\},\]
where \(\BT(Q_S)+\mathbf{u}\) means the graph obtained from \(\BT(Q_S)\) by translating it by \(\mathbf{u}\). Finally, write
\[\BT(Q_S,T,\Sigma_r=a)=\bigcup\left\{\BT(Q_S)+\mathbf{u}:\mathbf{u}\in Q_{T\setminus S}, \sum_{\alpha\in F}u_{r,\alpha}=a\right\},\]
\[\BT(Q_S,T,\Sigma_r=a,\Sigma_s=b)=\bigcup\left\{\BT(Q_S)+\mathbf{u}:\mathbf{u}\in Q_{T\setminus S}, \sum_{\alpha\in F}u_{r,\alpha}=a, \sum_{\alpha\in F}u_{s,\alpha}=b\right\}.\]

We now construct the junction rows using Claim \ref{clm:builder}, building \(Q_m\times \{\mathbf{w}_{\alpha,i}\}\) so that \(T_\alpha\left[Q_m\times\{\mathbf{w}_{\alpha,i}\}\right]\) includes \(G=G_i\) as follows:
\begin{align*}
G_1&=\BT\left(Q_{[2]\times F\setminus \{0\}}\right),\\
G_2&=\BT\left(Q_{\{1\}\times B \cup \{3\}\times F\setminus \{0\}},[2]\times F\setminus \{0\}\right),\\
G_3&=\BT\left(Q_{\{1\}\times B \cup \{4\}\times F\setminus \{0\}},[3]\times F\setminus \{0\}\right),\\
G_4&=\BT\left(Q_{\{1\}\times B \cup \{5\}\times F\setminus \{0\}},[4]\times F\setminus \{0\},\Sigma_3=0\right),\\
G_5&=\BT\left(Q_{\{1\}\times B \cup \{5\}\times F\setminus \{0\}},[4]\times F\setminus \{0\},\Sigma_3=1\right),\\
G_6&=\BT\left(Q_{\{1\}\times B \cup \{6\}\times F\setminus \{0\}},[5]\times F\setminus \{0\},\Sigma_3=0,\Sigma_4=0\right),\\
G_7&=\BT\left(Q_{\{1\}\times B \cup \{6\}\times F\setminus \{0\}},[5]\times F\setminus \{0\},\Sigma_3=0,\Sigma_4=1\right),\\
G_8&=\BT\left(Q_{\{1\}\times B \cup \{6\}\times F\setminus \{0\}},[5]\times F\setminus \{0\},\Sigma_3=1,\Sigma_4=0\right),\text{ and}\\
G_9&=\BT\left(Q_{\{1\}\times B \cup \{6\}\times F\setminus \{0\}},[5]\times F\setminus \{0\},\Sigma_3=1,\Sigma_4=1\right).\\
\end{align*}

Notice that these all meet the conditions of Claim \ref{clm:builder}, and that the use of \(B\) ensures that \(T_\alpha\) is a tree. Now each vertex \(\mathbf{u}\) of \(Q_m\) is connected in one of \(G_6,\dots,G_9\) to exactly one vertex \(\mathbf{u}'\) of \(Q_{[5]\times F\setminus \{0\}}\), namely the vertex which has \(u_{(i,\alpha)}=u'_{(i,\alpha)}\) for each \((i,\alpha)\in\left([5]\times F\setminus \{0\}\right)\setminus\left(\{1\}\times B\right)\). In particular, the distance between these two vertices in this \(G_i\) is at most \((2^k-1)+2k\).

By similar logic, we find that actually these \(G_i\) are sufficient to connect all the columns of label \(\alpha\) efficiently.

\begin{claim}
    In \(T_\alpha\), every vertex of \(Q_n\) has distance at most \(n+9k+11\) from \((\mathbf{0};\mathbf{w}_{\alpha,1})\).\label{clm:boring}
\end{claim}
\begin{proof}
    We describe a walk in \(T_\alpha\) from an arbitrary vertex \((\mathbf{u};\mathbf{w})\) to \((\mathbf{0};\mathbf{w}_{\alpha,1})\), in terms of \(\mathbf{u}_i\), which will all have label \(\alpha\).
    \begin{itemize}
        \item There is some \(\mathbf{u}_6\in Q_m\) of label \(\alpha\), which is either adjacent or equal to \(\mathbf{u}\). So the distance between \((\mathbf{u};\mathbf{w})\) and \((\mathbf{u}_6;\mathbf{w})\) is at most one.
        \item The column \(\mathbf{u}_6\times Q_{n-m}\) is a copy of \(T'\), and so \((\mathbf{u}_6;\mathbf{w})\) has distance at most \(n-m+k+2\) from each \((\mathbf{u}_6;\mathbf{w}_{\alpha,i})\).
        \item There is a path from \(\mathbf{u}_6\) to some \(\mathbf{u}_5\in Q_{[5]\times F\setminus \{0\}}\) in some \(G_i\) with \(i\in\{6,7,8,9\}\). This has length at most \((2^k-1)+2k\), and gives a path of that length in \(T_\alpha\) from \((\mathbf{u}_6;\mathbf{w}_{\alpha,i})\) to \((\mathbf{u}_5;\mathbf{w}_{\alpha,i})\).
        \item There is a path of length two from \((\mathbf{u}_5;\mathbf{w}_{\alpha,i})\) to any other \((\mathbf{u}_5;\mathbf{w}_{\alpha,j})\).
        \item There is a path from \(\mathbf{u}_5\) to some \(\mathbf{u}_4\in Q_{[4]\times F\setminus \{0\}}\) in some \(G_j\) with \(j\in\{4,5\}\). This has length at most \((2^k-1)+2k\), and gives a path of that length in \(T_\alpha\) from \((\mathbf{u}_5;\mathbf{w}_{\alpha,j})\) to \((\mathbf{u}_4;\mathbf{w}_{\alpha,j})\).
        \item There is a path of length two from \((\mathbf{u}_4;\mathbf{w}_{\alpha,j})\) to \((\mathbf{u}_4;\mathbf{w}_{\alpha,3})\).
        \item There is a path from \(\mathbf{u}_4\) to some \(\mathbf{u}_3\in Q_{[3]\times F\setminus \{0\}}\) in \(G_3\). This has length at most \((2^k-1)+2k\), and gives a path of that length in \(T_\alpha\) from \((\mathbf{u}_4;\mathbf{w}_{\alpha,3})\) to \((\mathbf{u}_3;\mathbf{w}_{\alpha,3})\).
        \item There is a path of length two from \((\mathbf{u}_3;\mathbf{w}_{\alpha,3})\) to \((\mathbf{u}_3;\mathbf{w}_{\alpha,2})\).
        \item There is a path from \(\mathbf{u}_3\) to some \(\mathbf{u}_2\in Q_{[2]\times F\setminus \{0\}}\) in \(G_2\). This has length at most \((2^k-1)+2k\), and gives a path of that length in \(T_\alpha\) from \((\mathbf{u}_3;\mathbf{w}_{\alpha,2})\) to \((\mathbf{u}_2;\mathbf{w}_{\alpha,2})\).
        \item There is a path of length two from \((\mathbf{u}_2;\mathbf{w}_{\alpha,2})\) to \((\mathbf{u}_2;\mathbf{w}_{\alpha,1})\).
        \item There is a path from \(\mathbf{u}_2\) to \(\mathbf{0}\) in \(G_1\). This has length at most \(2(2^k-1)\), and gives a path of that length in \(T_\alpha\) from \((\mathbf{u}_2;\mathbf{w}_{\alpha,1})\) to \((\mathbf{0};\mathbf{w}_{\alpha,1})\).
    \end{itemize}
    If, for example, some \(\mathbf{u}_i\) are not distinct, this may be a walk rather than a path, but it still gives an upper bound on the distance between \((\mathbf{u};\mathbf{w})\) and \((\mathbf{0};\mathbf{w}_{\alpha,1})\). But this walk has total length at most
    \[1+(n-m+k+2)+4((2^k-1)+2k)+4\cdot 2+2(2^k-1)=n+9k+11,\]
    as desired.
\end{proof}

But this implies that each tree \(T_\alpha\) has diameter at most \(2(n+9k+11)=\left(2+o(1)\right)n\). This concludes the proof of \ref{thm:secondtheorem}.

\section{Generalising to Cartesian powers}
The \emph{Cartesian product} of two graphs \(G_1,G_2\), written \(G_1\times G_2\), has vertex set \(V(G_1)\times V(G_2)\), and \((u,w)\) and \((u',w')\) adjacent if either \(u=u'\) and \(w\sim w'\), or \(u\sim u'\) and \(w=w'\). More generally, the Cartesian product \(G_1\times\cdots\times G_n\) has vertex set \(V(G_1)\times\dots\times V(G_n)\), with \(\mathbf{u}\sim\mathbf{u}'\) if and only if there is some \(i\in [n]\) s.t. \(u_i\sim u'_i\), and \(u_j=u'_j\) for all \(j\neq i\).

We will consider the \emph{Cartesian power} \(H^{n}\), defined as the Cartesian product of \(n\) copies of \(H\). Note that this is a natural generalisation of the hypercube. Indeed, we may write \(Q_n=K_2^{n}\), but also \(Q_{2n}=C_4^{n}\). We will need \(H\) to be connected for \(H^{n}\) to host spanning trees at all.

Take \(T^*\) to be a spanning tree of \(H\) with minimal diameter. Let \(z\in H\) be the centre of \(T^*\). We define a spanning tree of \(H^{n}\), the \emph{generalised broadcast tree} \(\BT(H^{n})\), by
\[E\left(\BT(H^{n})\right)=\left\{(u_1,\dots,u_{i-1},u_i,z,\dots,z)(u_1,\dots,u_{i-1},u'_i,z,\dots,z):i\in n, u_j\in H, u_iu'_i\in T^*\right\}.\]

We demonstrate that this is the best diameter we could hope to obtain. Where \(d_G(u,v)\) is the distance between \(u\) and \(v\) in the graph \(G\), we say a graph \(G\) has \emph{radius}
\[r(G)=\min_{u\in G}\left(\max_{v\in G}d_G(u,v)\right).\]
\begin{claim}
    The minimal diameter of a spanning tree of \(H^{n}\) is either \(2n\cdot r(H)-1\) or \(2n\cdot r(H)\).
\end{claim}
\begin{proof}
    First we show the minimal diameter of a spanning tree of \(H^{n}\) is at least \(2n\cdot r(H)-1\). Suppose \(T\) is a spanning tree of \(H^{n}\), and some vertex \(\mathbf{u}\) is a centre of \(T\). Then by the definition of \(r(H)\), for each \(i=1,\dots,n\) there is some \(v_i\in H\) with \(d_H(u_i,v_i)\geq r(H)\). Thus
    \[d_T(\mathbf{u},\mathbf{v})\geq d_{H^{n}}(\mathbf{u},\mathbf{v}) \geq n\cdot r(H).\]
    Hence the diameter of \(T\) is at least \(2n\cdot r(H) -1\).

    We now observe that \(\BT(H^{n})\) has diameter at most \(2n\cdot r(H)\).
    Consider the path from \(\mathbf{u}\) to \((\mathbf{z}=(z,\dots,z)\), which passes through \((u_1,\dots,u_{n-1},z),\dots,(u_1,z,\dots,z)\). Hence each \(\mathbf{u}\) has
    \[d_{\BT(H^{n})}\left(\mathbf{u},\mathbf{z}\right)=\sum_{i=1}^n d_{T^*}(u_i,z)\leq n\cdot r(H).\]
    Hence the tree has radius at most \(n\cdot r(H)\), and so has the desired diameter.
\end{proof}

We will prove the following result:

\begin{theorem}
    For \(k\) a positive integer, and \(n\geq 6\cdot 2^k + k + 4\), there exist \(2^k\) completely independent spanning trees in \(H^{n}\) of diameter at most \(2(n+5k+1)r(H)+4k+10=\left(2+o(1)\right)n\cdot r(H)\).\label{thm:thirdtheorem}
\end{theorem}

Let \(c:V(T^*)\to \mathbb{F}_2\) be the \(2\)-colouring of \(T^*\) with \(c(z)=0\). We extend the construction of Section 5 to \(H^{n}\). We will again regard \(H^{n}\) as \(H^{m}\times H^{(n-m)}\), and consider rows of the form \(H^{m}\times \{\mathbf{w}\}\) and columns of the form \(\{\mathbf{u}\}\times H^{(n-m)}\). We will still require \(n-m\geq k+10\).

We now say an edge \(\mathbf{u}\mathbf{u}'\in H^{n}\) has \emph{direction} \(i\) if \(u_j=u'_j\) for each \(j\neq i\), and \(u_i\sim_H u'_i\). We can no longer refer to \(\mathbf{e}_i\), so take \(n(\mathbf{u},i)\) to be a neighbour of \(\mathbf{u}\) by an edge in direction \(i\).

We again take \(m=6(2^k-1)\), and identify \([m]\) with \([6]\times F\setminus \{0\}\). We label the vertices of \(H^{m}\) by
\[\ell(\mathbf{u})=\sum_{(i,\alpha)\in[6]\times F\setminus\{0\}}c(u_{(i,\alpha)})\alpha.\]

\subsection{The columns}

We define \(T'\) for \(H^{n}\) analogously to Subsection \ref{subsec:lindiam-columns}. Let \(y\) be a leaf of \(T^*\), with neighbour \(x\). Set
\[T''=\BT(H^{(n-m)})\left[\left\{\mathbf{u}\in H^{(n-m)}:u_{n-m}\neq y\right\}\right].\] Choose distinct vertices \(\mathbf{w}_\alpha\in \mathbf{z}\times H^{k}\times \{y\}\) for each \(\alpha\) in \(F\). Now set \(W_\alpha\) to be the neighbours of \(\mathbf{w}_\alpha\) in directions from \([n-m-k-1]\times F\setminus \{0\}\). Notice that there are at least nine of these, as we have taken \(n-m\geq k+10\). Hence we can choose distinct \(\mathbf{w}_{\alpha,i}\in W_\alpha\) for \(i\in [9]\).

We take
\[T'=T''\cup \left\{\mathbf{w}_\alpha\mathbf{w}':\mathbf{w'}\in W_\alpha\right\}\cup \left\{(\mathbf{w}^-,x),(\mathbf{w}^-,y):\mathbf{w}^-\in H^{{(n-m-1)}}\setminus\cup_{\alpha\in F}W_\alpha\right\}.\]
Once again, each \(\mathbf{w}_\alpha\) is adjacent to at least nine leaves \(\mathbf{w}_{\alpha,i}\), and each vertex of \(H^{(n-m)}\) has distance at most \((n-m+k+1)r(H)\) from each \(\mathbf{w}_\alpha\).

\subsection{The rows}

We now generalise Claim \ref{clm:builder} to general Cartesian powers.

\begin{claim}
    Suppose \(G\subset H^{m}\) meets the following conditions:
    \begin{enumerate}
        \item \(G\) is a forest.
        \item For some \(d\), all the edges of \(G\) have a direction in \(\{1,d\}\times F\setminus \{0\}\).
        \item Every vertex with positive degree in \(G\) is connected to a vertex of label \(\alpha\).
        \item There exist \(r,s\in [6]\setminus\{1,b\}\) and \(a,b\in \mathbb{F}_2\) such that for any vertex \(\mathbf{u}\) of positive degree in \(G\),
        \[\sum_{\alpha\in F\setminus\{0\}}c(u_{r,\alpha})=a,\]
        and
        \[\sum_{\alpha\in F\setminus\{0\}}c(u_{s,\alpha})=b.\]
    \end{enumerate}
    Then at a leaf \(\mathbf{w}\) of \(T'\), we can construct a row \(H^{m}\times\{\mathbf{w}\}\) such that 
    \begin{enumerate}
        \item \(G\subset T_\alpha\left[H^{m}\times\{\mathbf{w}\}\right]\),
        \item For each \(\beta\), the induced subgraph \(T_\beta\left[H^{m}\times\{\mathbf{w}\}\right]\) is a forest in which every vertex of \(\mathbf{u}\) is connected to some vertex of label \(\beta\),
        \item the forests \(T_\beta\left[H^{m}\times\{\mathbf{w}\}\right]\) are edge-disjoint, and
        \item each vertex \(\mathbf{u}\) is in the interior of at most one resulting \(T_\beta\) on the entire \(H^{n}\). 
    \end{enumerate}\label{clm:generalbuilder}
\end{claim}
\begin{proof}
    Again, without loss of generality, assume that \(d=2,r=5\) and \(s=6\). Let \(V_0\) be the set of vertices of degree zero in \(G\) with labels other than \(\alpha\). Now set
    \[E\left(T_\alpha\left[H^{m}\times\{\mathbf{w}\}\right]\right)=E(G)\cup\left\{\left\{\mathbf{u},n(\mathbf{u},(1,\alpha+\ell(\mathbf{u})))\right\}:\mathbf{u}\in V_0\right\}.\]
    As before, none of the edges we add form a cycle, as we only add one edge to each isolated vertex of label other than \(\alpha\). So this is indeed a forest in which every vertex of \(\mathbf{u}\) is connected to some vertex of label \(\alpha\).
    
    We now construct the remaining \(T_\beta\left[H^{m}\times\{\mathbf{w}\}\right]\), in a similar way to before. For each \(\mathbf{u}\) with label other than \(\beta\), we will choose its one neighbour to be
    \begin{align*}
        n_{\beta}(\mathbf{u})={}&{}n\left(\mathbf{u},{(3,\beta+\ell(\mathbf{u}))}\right){}&{}\text{ if }\ell(\mathbf{u})<\beta\text{ and }\sum_{\alpha\in F\setminus\{0\}}c(u_{(5,\alpha)})\neq a,\\
        {}&{}n\left(\mathbf{u},{(4,\beta+\ell(\mathbf{u}))}\right){}&{}\text{ if }\beta<\ell(\mathbf{u})\text{ and }\sum_{\alpha\in F\setminus\{0\}}c(u_{(6,\alpha)})\neq b,\\
        {}&{}n\left(\mathbf{u},{(5,\beta+\ell(\mathbf{u}))}\right){}&{}\text{ if }\ell(\mathbf{u})<\beta\text{ and }\sum_{\alpha\in F\setminus\{0\}}c(u_{(5,\alpha)})=a,\\
        {}&{}n\left(\mathbf{u},{(6,\beta+\ell(\mathbf{u}))}\right){}&{}\text{ if }\beta<\ell(\mathbf{u})\text{ and }\sum_{\alpha\in F\setminus\{0\}}c(u_{(6,\alpha)})=b.\\
    \end{align*}

    Now again, for each \(\mathbf{u}\), the vertex \(n_{\beta}(\mathbf{u})\) has degree zero in \(G\), as we have
    \[\left(\sum_{\alpha\in F\setminus\{0\}}c(u_{(5,\alpha)}),\sum_{\alpha\in F\setminus\{0\}}c(u_{(6,\alpha)})\right)\neq (a,b).\]
    Hence for each \(\beta\), we take
    \[E\left(T_\beta\left[H^{m}\times\{\mathbf{w}\}\right]\right)=\left\{\mathbf{u}n_{\beta}(\mathbf{u}):\mathbf{u}\in H^{m}, \ell(\mathbf{u})\neq \beta \right\}.\]
    Now every edge of \(T_\beta\left[H^{m}\times\{\mathbf{w}\}\right]\) has an endpoint of degree one, so this is a forest. As before, the ordering of \(F\) ensures that all the \(T_\beta\left[H^{m}\times\{\mathbf{w}\}\right]\) are edge-disjoint. Thus we have met all our desired conditions.
\end{proof}

We define \(H^{S}\) analogously to \(Q_S\), and likewise define \(\BT(H^{S},T,\Sigma_r=a,\Sigma_s=b)\) analogously to \(\BT(Q_S,T,\Sigma_r=a,\Sigma_s=b)\). We will then repeat the construction of Section 5. We will have to be more careful in one place: previously we linked our \(G_i\) together using the fact that for every \(\mathbf{u}\in Q_m\), there is exactly one vertex \(\mathbf{u}'\in Q_m\) with label \(\alpha\) such that \(u_{i,\beta}=u'_{i,\beta}\) for every index \((i,\beta)\notin \{1\}\times B\). This is not necessarily true in \(H^{m}\), as \(c\) is not injective.

Instead, we fix \(x\) as some a neighbour of \(z\) in \(T^*\). Now we will use the fact that \(\mathbf{u}\in H^{m}\), there is exactly one vertex \(\mathbf{u}'\in H^{m}\) with label \(\alpha\) such that \(u_{i,\beta}=u'_{i,\beta}\) for every index \((i,\beta)\notin \{1\}\times B\), and \(u'_{1,\beta}\in\{x,z\}\) for each \(\beta\in B\). This will be sufficient to join our trees.

Explicitly, for \(S\subset [6]\times F\setminus \{0\}\), we write
\[H^{S}=\left\{\mathbf{u}\in H^{m}:u_{(i,\alpha)}=z\ \forall (i,\alpha)\in\left([6]\times F\setminus\{0\}\right)\setminus S\right\}.\]
Take \(\BT\left(H^{S}\right)\) to be the broadcast tree on \(H^{S}\), where we identify this with \(H^{|S|}\) in the natural way. For \(\{1\}\times B\subset S\), we now take
\[\BT^*\left(H^{S}\right)=\BT\left(H^{S}\right)\cap\left\{\mathbf{u}\in H^{S}:u_{(1,\beta)}\in \{x,z\}\ \forall \beta \in B\right\}.\]
Notice this still has diameter at most \(2|S| \cdot r(H)\). Indeed, it has slightly smaller diameter, as its radius is now \((|S|-k)\cdot r(H) + 1\).

For \(T\subset [6]\times F\setminus \{0\}\), we will take \(\BT\left(H^{S},T\right)\) to be the union of the copies of \(\BT\left(H^{S}\right)\) arising by replacing coordinates in \(T\setminus S\) with any vertex of \(Q_{T\setminus S}\). We define \(\BT^*\left(H^{S},T\right)\) analogously, taking
\[\BT^*\left(H^{S},T\right)=\BT\left(H^{S},T\right)\cap\left\{\mathbf{u}\in H^{S}:u_{(1,\beta)}\in \{x,z\}\ \forall \beta \in B\right\}.\]
For \(S,T\subset ([6]\setminus \{r\})\times F\setminus \{0\}\), we set
\[\BT\left(H^{S},T,\Sigma_r=a\right)=\left\{\mathbf{u}\mathbf{u}'\in \BT\left(H^{S},T\right):\sum_{\alpha \in F}c(u_{r,\alpha})=a\right\}.\]
We define \(\BT^*\left(H^{S},T,\Sigma_r=a\right)\), \(\BT\left(H^{S},T,\Sigma_r=a,\Sigma_s=b\right)\), and 
\(\BT^*\left(H^{S},T,\Sigma_r=a,\Sigma_s=b\right)\) analogously.

We construct the junction rows in the same way as before, using Claim \ref{clm:generalbuilder}. We will build row \(Q_m\times \{\mathbf{w}_{\alpha,i}\}\) so that \(T_\alpha\left[Q_m\times\{\mathbf{w}_{\alpha,i}\}\right]\) includes \(G=G_i\) as follows:
\begin{align*}
G_1&=\BT^*\left(H^{[2]\times F\setminus \{0\}}\right),\\
G_2&=\BT^*\left(H^{\{1\}\times B \cup \{3\}\times F\setminus \{0\}},[2]\times F\setminus \{0\}\right),\\
G_3&=\BT^*\left(H^{\{1\}\times B \cup \{4\}\times F\setminus \{0\}},[3]\times F\setminus \{0\}\right),\\
G_4&=\BT^*\left(H^{\{1\}\times B \cup \{5\}\times F\setminus \{0\}},[4]\times F\setminus \{0\},\Sigma_3=0\right),\\
G_5&=\BT^*\left(H^{\{1\}\times B \cup \{5\}\times F\setminus \{0\}},[4]\times F\setminus \{0\},\Sigma_3=1\right),\\
G_6&=\BT\left(H^{\{1\}\times B \cup \{6\}\times F\setminus \{0\}},[5]\times F\setminus \{0\},\Sigma_3=0,\Sigma_4=0\right),\\
G_7&=\BT\left(H^{\{1\}\times B \cup \{6\}\times F\setminus \{0\}},[5]\times F\setminus \{0\},\Sigma_3=0,\Sigma_4=1\right),\\
G_8&=\BT\left(H^{\{1\}\times B \cup \{6\}\times F\setminus \{0\}},[5]\times F\setminus \{0\},\Sigma_3=1,\Sigma_4=0\right),\text{ and}\\
G_9&=\BT\left(H^{\{1\}\times B \cup \{6\}\times F\setminus \{0\}},[5]\times F\setminus \{0\},\Sigma_3=1,\Sigma_4=1\right).\\
\end{align*}

Our use of \(B\) with \(\BT^*\) ensures that \(T_\alpha\) is in fact a tree. Similar reasoning to Claim \ref{clm:boring} implies that in \(T_\alpha\), every vertex has distance from \((\mathbf{z},\mathbf{w}_{\alpha,1})\) at most
\[1+(n-m+k+1) r(H)+1 + 4\cdot\left((2^k+k-1)r(H)+k\right) +4\cdot 2+2\left(2^k-1\right)r(H).\]
But this is at most \((n + 5 k + 1)r(H) + 4k + 10\). Then each \(T_\alpha\) has diameter at most twice this, which concludes the proof of Theorem \ref{thm:thirdtheorem}.

\section{Further questions}

We have shown that there exist \(\left(\frac1{12}+o(1)\right)n\) completely independent spanning trees in \(Q_n\). Consideration of the number of edges in each spanning tree shows that we cannot find a set of more than \(\left(\frac12+o(1)\right)n\) completely independent spanning trees in \(Q_n\). This motivates the following question:

\begin{question}
    What is the greatest constant \(c\) such that, for each \(n\), there exist \((c+o(1))n\) completely independent spanning trees of \(Q_n\)?
\end{question}

Note that our result does not construct exactly \(\left(\frac1{12}+o(1)\right)n\) completely independent spanning trees of \(Q_n\), as for certain \(n\) it constructs \(\left(\frac1{6}+o(1)\right)n\). Writing \(f(n)\) for the maximum size of a set of completely independent spanning trees of \(Q(n)\), we may ask

\begin{question}
    As \(n\to\infty\), does \(\frac{f(n)}n\) tend to a limit?
\end{question}

We may also pose either of these questions with the additional constraint that the spanning trees have diameter \((2+o(1))n\).

\bibliographystyle{plain}
\bibliography{hypercube-cists}

\begin{thebibliography}{10}

\bibitem{araki_diracs_2014}
Toru Araki.
\newblock Dirac's {Condition} for {Completely} {Independent} {Spanning} {Trees}.
\newblock {\em Journal of Graph Theory}, 77(3):171--179, 2014.

\bibitem{cheng_independent_2023}
Baolei Cheng, Dajin Wang, and Jianxi Fan.
\newblock Independent {Spanning} {Trees} in {Networks}: {A} {Survey}.
\newblock {\em ACM Computing Surveys}, 55(14s):335:1--29, 2023.

\bibitem{hasunuma_completely_2001}
Toru Hasunuma.
\newblock Completely independent spanning trees in the underlying graph of a line digraph.
\newblock {\em Discrete Mathematics}, 234(1):149--157, 2001.

\bibitem{hasunuma_completely_2002}
Toru Hasunuma.
\newblock Completely {Independent} {Spanning} {Trees} in {Maximal} {Planar} {Graphs}.
\newblock In {\em Graph-{Theoretic} {Concepts} in {Computer} {Science}}, pages 235--245, 2002.

\bibitem{hong_completely_2018}
Xia Hong.
\newblock Completely independent spanning trees in \emph{k}-th power of graphs.
\newblock {\em Discussiones Mathematicae Graph Theory}, 38(3):801--810, 2018.

\bibitem{kandekar_constructing_2024}
S.~A. Kandekar and S.~A. Mane.
\newblock Constructing three completely independent spanning trees in hypercubes, 2024.
\newblock arXiv:2410.03379.

\bibitem{pai_constructing_2016}
Kung-Jui Pai and Jou-Ming Chang.
\newblock Constructing two completely independent spanning trees in hypercube-variant networks.
\newblock {\em Theoretical Computer Science}, 652:28--37, 2016.

\bibitem{pai_three_2020}
Kung-Jui Pai, Ruay-Shiung Chang, Ro-Yu Wu, and Jou-Ming Chang.
\newblock Three completely independent spanning trees of crossed cubes with application to secure-protection routing.
\newblock {\em Information Sciences}, 541:516--530, 2020.

\bibitem{pai_completely_2014}
Kung-Jui Pai, Jinn-Shyong Yang, Sing-Chen Yao, Shyue-Ming Tang, and Jou-ming Chang.
\newblock Completely {Independent} {Spanning} {Trees} on {Some} {Interconnection} {Networks}.
\newblock {\em IEICE Transactions on Information and Systems}, E97.D:2514--2517, 2014.

\bibitem{peterfalvi_two_2012}
Ferenc Péterfalvi.
\newblock Two counterexamples on completely independent spanning trees.
\newblock {\em Discrete Mathematics}, 312(4):808--810, 2012.

\bibitem{tang_optimal_2004}
Shyue-Ming Tang, Yue-Li Wang, and Yungho Leu.
\newblock Optimal {Independent} {Spanning} {Trees} on {Hypercubes}.
\newblock {\em Journal of Information Science and Engineering}, 20:143--156, 2004.

\end{thebibliography}

\end{document}